\documentclass{amsart}

\usepackage{bbm} 
\usepackage{nicefrac}
\usepackage{hyperref}

\newcommand*{\mailto}[1]{\href{mailto:#1}{\nolinkurl{#1}}}

\newtheorem{theorem}{Theorem}[section]

\newtheorem{lemma}[theorem]{Lemma}
\newtheorem{corollary}[theorem]{Corollary}

\newcommand{\R}{{\mathbb R}}
\newcommand{\N}{{\mathbb N}}

\newcommand{\C}{{\mathbb C}}

\newcommand{\spr}[2]{\langle #1 , #2 \rangle}
\newcommand{\dbspr}[2]{[ #1 , #2 ]}
\newcommand{\E}{\mathrm{e}}
\newcommand{\I}{\mathrm{i}}

\newcommand{\LdBsm}{L^2(\R;\rho)}

\newcommand{\oo}{o}

\renewcommand{\labelenumi}{(\roman{enumi})}
\numberwithin{equation}{section}

\begin{document}

\title[Inverse uniqueness results using de Branges theory]{Inverse uniqueness results for Schr\"{o}dinger operators using~de~Branges theory}

\author[J.\ Eckhardt]{Jonathan Eckhardt}
\address{Faculty of Mathematics\\ University of Vienna\\
Nordbergstrasse 15\\ 1090 Wien\\ Austria}
\email{\mailto{jonathan.eckhardt@univie.ac.at}}
\urladdr{\url{http://homepage.univie.ac.at/jonathan.eckhardt/}}

\thanks{\href{http://dx.doi.org/10.1007/s11785-012-0265-3}{Complex Anal.\ Oper.\ Theory {\bf 8} (2014), no.~1, 37--50}}
\thanks{{\it Research supported by the Austrian Science Fund (FWF) under Grant No.\ Y330}}

\keywords{Schr\"odinger operators, de Branges spaces, strongly singular potentials}
\subjclass[2010]{Primary 34L05, 46E22; Secondary 34L40, 34B24.}

\begin{abstract}
We utilize the theory of de Branges spaces to show when certain Schr\"{o}dinger operators with strongly singular potentials are uniquely determined by their associated spectral measure.
 The results are applied to obtain an inverse uniqueness theorem for perturbed spherical Schr\"odinger operators.
\end{abstract}

\maketitle

\section{Introduction}

We consider Schr\"{o}dinger operators $H$ (with separated boundary conditions), associated with the differential expression
\begin{align*}
 \tau = -\frac{d^2}{dx^2} + q(x)
\end{align*}
on some interval $(a,b)$, where $q\in L^1_{\mathrm{loc}}(a,b)$ is a real-valued potential. 
 It has been shown by Kodaira \cite{ko}, Kac \cite{ka} and more recently by Fulton \cite{ful2}, Gesztesy and Zinchenko \cite{geszin}, Fulton and Langer \cite{fullan}, Kurasov and Luger \cite{kurlug}, and Kostenko, Sakhnovich and Teschl \cite{kt}, \cite{kst}, \cite{kst2} that, even when the potential is quite singular at $a$, it is still possible to introduce a singular Weyl--Titchmarsh function as well as a scalar spectral measure. Indeed, this only requires some nontrivial real entire solution $\phi$ of
\begin{align*}
 -\phi''(z,x) + q(x)\phi(z,x) = z\phi(z,x), \quad x\in(a,b),~z\in\C,
\end{align*}
which lies in $L^2(a,b)$ near $a$ and satisfies the boundary condition at $a$ if $\tau$ is in the limit-circle case there. Here, by a real entire solution we mean that the functions 
\begin{align*}
 z\mapsto\phi(z,c) \quad\text{and}\quad z\mapsto\phi'(z,c)
\end{align*}
 are real entire for one (and hence for all) $c\in(a,b)$.
 For example, if $\tau$ is in the limit-circle case at $a$, then such a solution is known to exist and Weyl--Titchmarsh theory has been developed e.g.\ in~\cite{ful}, \cite{beneve}, analogously to the regular case. 
 In general, for such a solution $\phi$ to exist it is necessary and sufficient that the operator $H_c$ has purely discrete spectrum (see e.g.\ \cite[Lemma~3.2]{geszin}, \cite[Lemma~2.2]{kst2}) for some $c\in(a,b)$, where $H_c$ is the restriction of $H$ to $L^2(a,c)$ with some additional self-adjoint boundary condition at the point $c$.

Regarding inverse spectral theory, Kostenko, Sakhnovich and Teschl~\cite{kst2} were able to prove a local Borg--Marchenko uniqueness result for the singular Weyl--Titchmarsh function under restrictions on the exponential growth of solutions. 
 Their proof follows the simple proof of Bennewitz~\cite{ben}, which covers the case of regular left endpoints.
However, since the spectral measure determines the singular Weyl--Titchmarsh function only up to some real entire function,
 their Borg--Marchenko theorem does not immediately yield an inverse uniqueness result for the associated spectral measure.
 In fact, all one would need is some growth restriction on the difference of two singular Weyl--Titchmarsh functions, corresponding to the same spectral measure.  This has been done in~\cite{SingWT2}, for the case when the spectra of the operators are assumed to be purely discrete with finite convergence exponent. 
 
The present paper pursues a different approach. 
 We utilize de Branges' theory of Hilbert spaces of entire functions in order to obtain an inverse uniqueness theorem for the spectral measure.
 More precisely, we apply de Branges' subspace ordering theorem to conclude that the de Branges spaces associated with Schr\"{o}dinger operators with a common spectral measure are equal.
 Therefore, we will first provide a brief review of the theory of de Branges spaces in Section~\ref{secdB}.
For a detailed discussion we refer to de Branges' book~\cite{dBbook}.
The following section introduces de Branges spaces associated with a self-adjoint Schr\"{o}dinger operator as above. 
The core of this section is quite similar to~\cite[Section~3]{remling} (see also~\cite{remling2}) with the only difference that we do not assume the left endpoint to be regular.
Section~\ref{secdBuniq} is devoted to our inverse uniqueness theorem for the spectral measure. 
Finally,  we apply our results to perturbed spherical Schr\"{o}dinger (or Bessel) operators. 

As a last remark, let us mention that the approach taken here equally well applies to more general operators. 
 For example consider the differential expression 
\begin{align*}
  - \frac{d}{dx} \frac{d}{d\varsigma(x)} + \chi(x)
\end{align*}
 on some interval $(a,b)$, where $\varsigma$ and $\chi$ are real-valued Borel measures on $(a,b)$ which do not have common point masses, that is, $\varsigma(\lbrace x\rbrace) \chi(\lbrace x\rbrace) = 0$ for all $x\in(a,b)$. 
 Moreover, we suppose that the absolutely continuous part of $\varsigma$ is Lebesgue measure.  
 This differential expression (together with possible separated boundary conditions) gives rise to a unique self-adjoint operator in $L^2(a,b)$ (see \cite{measureSL} for details or \cite{ka}, \cite{bamrem} for the case when $\varsigma$ is Lebesgue measure). 
  For example, this kind of operators include Schr\"odinger operators with local $\delta$ and $\delta'$ point interactions on discrete sets as well as ones with (non-local) point interactions on sets of Lebesgue measure zero  as studied recently in \cite{albniz}, \cite{albniz2}, \cite{braniz}.  
 Section~\ref{secSdB} literally holds for these more general differential expressions as well (see \cite{measureSL} for the spectral theory and \cite[Corollary~6.2]{ben2} for the necessary high energy asymptotics). Moreover, the inverse uniqueness results in Section~\ref{secdBuniq} are literally the same ($\varsigma$ and $\chi$ are both determined up to some shift) with some minor modifications in the second half of the proof of Theorem~\ref{thmdBuniqS}. 
 Finally, note that our approach also applies to general self-adjoint Sturm--Liouville operators associated with differential expressions of the form 
 \begin{align*}
  - \frac{1}{r(x)}\frac{d}{dx} \frac{1}{s(x)} \frac{d}{dx} + \frac{q(x)}{r(x)}
 \end{align*}
 with three coefficient functions on some interval $(a,b)$. 
 In that case, the associated operators are determined by the spectral measure only up to a so-called Liouville transform as shown in \cite{ben1} (see also \cite{LeftDefiniteSL}).

\section{Hilbert spaces of entire functions}\label{secdB}

First of all, recall that an analytic function $N$ in the open upper complex half-plane $\C^+$ is said to be of bounded type if it can be written as the quotient of two bounded analytic functions.
For such a function, the quantity
\begin{align*}
 \limsup_{y\rightarrow\infty} \frac{\ln|N(\I y)|}{y} \in[-\infty,\infty)
\end{align*}
is referred to as the mean type of $N$.

A de Branges function is an entire function $E$, which satisfies the estimate
\begin{align*}
 |E(z)| > |E(z^\ast)|, \quad z\in\C^+.
\end{align*}
The de Branges space $B$ associated with such a function consists of all entire functions $F$ such that
\begin{align*}
 \int_\R \frac{|F(\lambda)|^2}{|E(\lambda)|^2} d\lambda < \infty
\end{align*}
and such that $F/E$ and $F^\#/E$ are of bounded type in $\C^+$ with non-positive mean type.
Here $F^\#$ is the entire function given by
\begin{align*}
 F^\#(z) = F(z^\ast)^\ast, \quad z\in\C.
\end{align*}
Equipped with the inner product
\begin{align*}
 \dbspr{F}{G} = \frac{1}{\pi} \int_\R \frac{F(\lambda)G(\lambda)^\ast}{|E(\lambda)|^2} d\lambda, \quad F,\,G\in B,
\end{align*}
 the vector space $B$ turns into a Hilbert space (see~\cite[Theorem~21]{dBbook}).
 For each $\zeta\in\C$, the point evaluation in $\zeta$ is a continuous linear functional on $B$, that is,
 \begin{align*}
  F(\zeta) = \dbspr{F}{K(\zeta,\cdot\,)}, \quad F\in B,
 \end{align*}
 where the reproducing kernel $K$ is given by (see~\cite[Theorem~19]{dBbook})
 \begin{align}\label{eqndBrepker}
  K(\zeta,z) = \frac{E(z)E^\#(\zeta^\ast)-E(\zeta^\ast)E^\#(z)}{2\I (\zeta^\ast-z)}, \quad \zeta,\,z\in\C.
 \end{align}
 Hereby note that though there is a multitude of de Branges functions giving rise to the same de Branges space (including norms), the reproducing kernel $K$ is independent of the actual de Branges function. 
 
Our uniqueness result relies on the subspace ordering theorem due to de Branges;\ \cite[Theorem~35]{dBbook}. 
 In order to state it, let $E_1$, $E_2$ be two de Branges functions and $B_1$, $B_2$ be the corresponding de Branges spaces.

\begin{theorem}\label{thmdBOrdering}
Suppose $B_1$, $B_2$ are isometrically embedded in $\LdBsm$ for some Borel measure $\rho$ on $\R$. 
If $E_1/E_2$ is of bounded type in the upper com\-plex half-plane and has no real zeros or singularities, then $B_1$ contains $B_2$ or $B_2$ contains $B_1$.
\end{theorem}

Moreover, one has the following simple converse statement.

\begin{lemma}\label{lemdBordcon}
If $B_1$ contains $B_2$ or $B_2$ contains $B_1$, then $E_1/E_2$ is of bounded type in the upper complex half-plane.
\end{lemma}

\begin{proof}
 For each $F\in B_1\cap B_2$ the quotients $F/E_1$ and $F/E_2$ are of bounded type in the upper complex half-plane by definition, which implies the claim.
\end{proof}

\section{Schr\"{o}dinger operators with strongly singular potentials}\label{secSdB}

In this section let $(a,b)$ be some bounded or unbounded interval, $q$ be a real-valued, locally integrable function on $(a,b)$ and $\tau$ be the differential expression
\begin{align*}
 \tau = -\frac{d^2}{dx^2} + q(x)
\end{align*} 
on $(a,b)$.
With $H$ we denote some associated self-adjoint Schr\"{o}dinger operator in $L^2(a,b)$ with separated boundary conditions (if $\tau$ is in the limit-circle case at both endpoints).
Concerning the regularity of $\tau$ near the endpoint $a$, we will only assume that there is some real entire solution $\phi$ of
\begin{align*}
 -\phi''(z,x) + q(x)\phi(z,x) = z\phi(z,x), \quad x\in(a,b),~z\in\C,
\end{align*}
such that for each $z\in\C$, $\phi(z,\cdot\,)$ is not identically zero, lies in $L^2(a,b)$ near $a$ and satisfies the boundary condition at $a$ if $\tau$ is in the limit-circle case there. 
Here, by real entire we mean that for some (and hence for all) $c\in(a,b)$ the functions 
\begin{align*}
z\mapsto\phi(z,c) \quad\text{and}\quad z\mapsto\phi'(z,c)
\end{align*}
 are real entire.
%In~\cite{kst2} it has been shown that a necessary and sufficient condition for such a solution to exist is, that $H$ has purely discrete spectrum near $a$.
  For the proof of our inverse uniqueness result we will need the following high energy asymptotics of the solution $\phi$, which may be deduced from the asymptotics in \cite[Section~9.4]{tschroe}.
 Note that we always use the principal square root with branch cut along the negative real axis.

\begin{lemma}\label{lemSchrPhiAsym}
 For every $x$, $\tilde{x}\in(a,b)$ we have the asymptotics
 \begin{align}
   \frac{\phi(z,x)}{\phi(z,\tilde{x})} = \E^{\sqrt{-z}(x-\tilde{x}+\oo(1))}
 \end{align}
 as $|z|\rightarrow\infty$ along the imaginary axis.
\end{lemma}

%\begin{proof}
% For each $z\in\C$ let $c(z,\cdot\,)$ and $s(z,\cdot\,)$ be the solutions of $(\tau-z)u=0$ with the initial conditions
% \begin{align*}
%   c(z,\tilde{x}) = s'(z,\tilde{x}) = 1 \quad\text{and}\quad c'(z,\tilde{x})=s(z,\tilde{x})=0.
% \end{align*}
% Now if $x\geq \tilde{x}$ the claim follows from
% \begin{align*}
%  \phi(z,x) = \phi(z,\tilde{x}) \left(c(z,x) + \frac{\phi'(z,\tilde{x})}{\phi(z,\tilde{x})} s(z,x)\right), \quad z\in\C\backslash\R
% \end{align*}
% and the well-known asymptotics of the quotient on the right-hand side (see~\cite[Lemma~9.19]{tschroe}) and the solutions $c$ and $s$ (see~\cite[Lemma~9.18]{tschroe}).
% The case when $x<\tilde{x}$ follows by reversing the roles of $x$ and $\tilde{x}$.
%\end{proof}

Given some $c\in(a,b)$, we denote with $L^2(a,c)$ the closed linear subspace of $L^2(a,b)$ consisting of all functions which vanish outside of $(a,c)$ almost everywhere. 
Now as in the case of regular left endpoints, one may define the transform of a function $f\in L^2(a,c)$ as
\begin{align}\label{eqndBftrans}
 \hat{f}(z) = \int_a^b \phi(z,x)f(x) dx, \quad z\in\C.
\end{align} 
Given this, it is known (see e.g.\ \cite[Section~3]{geszin}, \cite[Section~3]{kst2}) that there is some Borel measure $\rho$ on $\R$ such that 
\begin{align*}%\label{eqndBIsotrans}
 \int_\R |\hat{f}(\lambda)|^2 d\rho(\lambda) = \int_a^b |f(x)|^2 dx, \quad f\in L^2(a,c)
\end{align*}
holds for all $c\in(a,b)$.
Moreover, this transformation uniquely extends to a unitary map from $L^2(a,b)$ onto $\LdBsm$ and the operator $H$ is mapped onto multiplication with the independent variable in $\LdBsm$. 
Note that the measure $\rho$ is uniquely determined by these properties and hence referred to as the spectral measure of $H$ associated with the solution $\phi$.
% Each measure with this properties satisfies $E_{f,g} = \hat{f} \hat{g}^\ast \mu$

From these results, one sees that the space of transforms of all functions in $L^2(a,c)$, equipped with the norm inherited from the space $\LdBsm$, forms a Hilbert space.
In order to show that it is even a de Branges space, fix some $c\in(a,b)$ and consider the entire function
\begin{align*}%\label{eqndBschrE}
 E(z,c) = \phi(z,c) + \I \phi'(z,c), \quad z\in\C.
\end{align*}
Using the Lagrange identity and the fact that the Wronskian of two solutions satisfying the same boundary condition at $a$ (if any) vanishes in $a$, one gets
%\begin{align*}
% \lim_{x\rightarrow a} \phi(z,x)\phi'(\zeta,x) - \phi'(z,x)\phi(\zeta,x) = 0, \quad \zeta,\,z\in\C,
%\end{align*}
%(note that both functions satisfy the same boundary condition near $a$), one immediately gets 
\begin{align*} 
 \frac{E(z,c) E^\#(\zeta^\ast,c) - E(\zeta^\ast,c) E^\#(z,c)}{2\I (\zeta^\ast -z)} = \int_a^c \phi(\zeta,x)^\ast \phi(z,x) dx, \quad \zeta,\,z\in\C^+.
\end{align*}
In particular, taking $\zeta=z$ this shows that $E(\,\cdot\,,c)$ is a de Branges function.
Moreover, note that $E(\,\cdot\,,c)$ does not have any real zero $\lambda$, since otherwise both, $\phi(\lambda,c)$ and $\phi'(\lambda,c)$ would vanish.
With $B(c)$ we denote the de Branges space associated with the de Branges function $E(\,\cdot\,,c)$ endowed with the inner product
\begin{align*}
 \dbspr{F}{G}_{B(c)} = \frac{1}{\pi} \int_\R \frac{F(\lambda) G(\lambda)^\ast}{|E(\lambda,c)|^2} d\lambda = \frac{1}{\pi}\int_\R \frac{F(\lambda)G(\lambda)^\ast}{\phi(\lambda,c)^2 + \phi'(\lambda,c)^2}d\lambda, \quad F,\, G\in B(c).
\end{align*}
Now using~\eqref{eqndBrepker} and a similar calculation as above, one shows that the reproducing kernel $K(\,\cdot\,,\cdot\,,c)$ of this space is given by
\begin{align}\label{eqndBschrRepKer}
 K(\zeta,z,c) = \int_a^c \phi(\zeta,x)^\ast \phi(z,x)dx, \quad \zeta,\,z\in\C.
\end{align}

\begin{theorem}\label{thmdBschrBT}
 For every $c\in(a,b)$ the transformation $f\mapsto\hat{f}$ is unitary from $L^2(a,c)$ onto $B(c)$, in particular
 \begin{align}
  B(c) = \big\lbrace  \hat{f} \,\big|\, f\in L^2(a,c) \big\rbrace.
 \end{align}
\end{theorem}
 
\begin{proof}
 For each $\lambda\in\R$ consider the function
 \begin{align*}
  f_\lambda(x) = \begin{cases} \phi(\lambda,x), & x\in(a,c], \\
                               0,               & x\in(c,b). \end{cases}
 \end{align*}
 The transforms of these functions are given by
 \begin{align*}
  \hat{f}_\lambda(z) = \int_a^c \phi(\lambda,x) \phi(z,x)dx = K(\lambda,z,c), \quad z\in\C.
 \end{align*}
 In particular, this shows that the transforms of the functions $f_\lambda$, $\lambda\in\R$ lie in $B(c)$. Moreover, we have for all $\lambda_1$, $\lambda_2\in\R$
 \begin{align*}
  \spr{f_{\lambda_1}}{f_{\lambda_2}}  = \int_a^c \phi(\lambda_1,x) \phi(\lambda_2,x) dx = K(\lambda_1,\lambda_2,c)                 = \dbspr{K(\lambda_1,\cdot\,,c)}{K(\lambda_2,\cdot\,,c)}_{B(c)}. %\\ & = \dbspr{\hat{f}_{\lambda_1}}{\hat{f}_{\lambda_2}}_{B(c)},
 \end{align*}
 Hence our transform is an isometry on the linear span $D$ of all functions $f_\lambda$, $\lambda\in\R$.
  But this span is dense in $L^2(a,c)$ since it contains the eigenfunctions of the restricted operator $H_c$ in $L^2(a,c)$.
 Moreover, the linear span of all transforms $K(\lambda,\cdot\,,c)$, $\lambda\in\R$ is dense in $B(c)$. Indeed, each $F\in B(c)$ such that
 \begin{align*}
  0 = \dbspr{F}{K(\lambda,\cdot\,,c)}_{B(c)} = F(\lambda), \quad \lambda\in\R
 \end{align*}
 vanishes identically. Thus our transformation restricted to $D$ uniquely extends to a unitary map $V$ from $L^2(a,c)$ onto $B(c)$.
 In order to identify $V$ with our transformation, note that for each fixed $z\in\C$, both $f\mapsto\hat{f}(z)$ and $f\mapsto Vf(z)$ are continuous on $L^2(a,c)$.
\end{proof}

As an immediate consequence of Theorem~\ref{thmdBschrBT} and the fact that our transformation from~\eqref{eqndBftrans} extends 
 to a unitary map from $L^2(a,b)$ onto $\LdBsm$, we get the following corollary.

\begin{corollary}\label{cordBschrembL2}
For each $c\in(a,b)$ the de Branges space $B(c)$ is isometrically embedded in $\LdBsm$, that is, 
\begin{align}
 \int_\R |F(\lambda)|^2 d\rho(\lambda) = \|F\|^2_{B(c)}, \quad F\in B(c).
\end{align}
Moreover, the union of the spaces $B(c)$, $c\in(a,b)$ is dense in $\LdBsm$ 
\begin{align}\label{eqndBschrdense}
 \overline{\bigcup_{c\in(a,b)} B(c)} = \LdBsm.
\end{align}
\end{corollary}

The following corollary shows that the de Branges spaces $B(c)$, $c\in(a,b)$ are totally ordered, strictly increasing and continuous in some sense.

\begin{corollary}\label{cordBschrincl}
If $c_1$, $c_2\in(a,b)$ with $c_1<c_2$, then $B(c_1)$ is isometrically embedded in, but not equal to $B(c_2)$.
Moreover, for each $c\in(a,b)$ we have
\begin{align}\label{eqndBschrcontinuous}
 \overline{\bigcup_{x\in(a,c)} B(x)} = B(c) = \bigcap_{x\in(c,b)}B(x).
\end{align}
\end{corollary}

\begin{proof}
 The embedding is clear from Theorem~\ref{thmdBschrBT} and Corollary~\ref{cordBschrembL2}. 
  Moreover, Theorem~\ref{thmdBschrBT} shows that $B(c_2)\,\ominus\, B(c_1)$ is unitarily equivalent to the space $L^2(c_1,c_2)$, hence $B(c_1)$ is not equal to $B(c_2)$.
  The second claim follows from the similar fact 
  \begin{align*}
   \overline{\bigcup_{x\in(a,c)} L^2(a,x)} = L^2(a,c) = \bigcap_{x\in(c,b)} L^2(a,x).
  \end{align*}
\end{proof}

As a final remark, note that the solution $\phi$ is not uniquely determined. 
 In fact,~\cite[Corollary~2.3]{kst2} shows that any other solution with the same properties as $\phi$ is given by
 \begin{align*}
  \tilde{\phi}(z,x) = \E^{g(z)} \phi(z,x), \quad x\in(a,b),~z\in\C,
 \end{align*}
 where $\E^g$ is some real entire function without zeros.
 Furthermore, \cite[Remark~3.8]{kst2} shows that the corresponding spectral measures are related by
 \begin{align*}
  \tilde{\rho} = \E^{-2g} \rho.
 \end{align*}
 In particular, they are mutually absolutely continuous.
 Using Theorem~\ref{thmdBschrBT}, it is easily seen that for each $c\in(a,b)$, multiplication with the entire function $\E^g$ maps $B(c)$ isometrically onto the corresponding de Branges space $\tilde{B}(c)$.

\section{Uniqueness of the inverse problem}\label{secdBuniq}

We will now prove our inverse uniqueness result.
Therefore, let $q_1$, $q_2$ be two real-valued, locally integrable functions on intervals $(a_1,b_1)$ respectively  $(a_2,b_2)$ and 
 $H_1$, $H_2$ be two associated self-adjoint Schr\"{o}dinger operators with separated boundary conditions.
 As in the previous section, we suppose that there are nontrivial real entire solutions $\phi_1$, $\phi_2$ which are square integrable near the left endpoint and satisfy the respective boundary condition there, if any.
 All remaining quantities corresponding to $H_1$ and $H_2$ are denoted in an obvious way with an additional subscript. 
 %we denote with $\rho_1$, $\rho_2$ the corresponding spectral measures, with $E_1$, $E_2$ the corresponding de Branges functions, 
 %with $B_1$, $B_2$ the corresponding de Branges spaces and with $K_1$, $K_2$ the corresponding reproducing kernels.
 We say $H_1$ and $H_2$ are equal up to some shift if there is a linear function $\eta$ with $\eta'=1$, mapping $(a_1,b_1)$ onto $(a_2,b_2)$ such that $q_1=q_2\circ\eta$ and 
\begin{align*}
H_1=U^{-1}H_2\,U,
\end{align*}
 where $U$ is the unitary map from $L^2(a_1,b_1)$ onto $L^2(a_2,b_2)$ induced by $\eta$.

\begin{theorem}\label{thmdBuniqS}
  Suppose there is a real entire function $\E^g$ without zeros such that
  \begin{align}\label{eqnquotE1E2}
   \E^{g(z)} \frac{E_1(z,x_1)}{E_2(z,x_2)}, \quad z\in\C^+
  \end{align}
  is of bounded type for some $x_1\in (a_1,b_1)$ and $x_2\in(a_2,b_2)$.
  If $\rho_1=\E^{-2g} \rho_2$, then $H_1$ and $H_2$ are equal up to some shift.
\end{theorem}

\begin{proof}
First of all note that without loss of generality we may assume that $g$ vanishes identically, since otherwise we replace $\phi_1$ with $\E^g \phi_1$.
Moreover, because of Lemma~\ref{lemdBordcon} the function in~\eqref{eqnquotE1E2} is of bounded type for all points $x_1\in(a_1,b_1)$ and $x_2\in(a_2,b_2)$.
Now fix some arbitrary $x_1\in(a_1,b_1)$. 
Since for each $x_2\in(a_2,b_2)$, both $B_1(x_1)$ and $B_2(x_2)$ are isometrically contained in $L^2(\R;\rho_1)$, we infer from Theorem~\ref{thmdBOrdering} 
 (note that~\eqref{eqnquotE1E2} has no real zeros or singularities because $E_1(\,\cdot\,,x_1)$ and $E_2(\,\cdot\,,x_2)$ do not have real zeros)
 that $B_1(x_1)$ is contained in $B_2(x_2)$ or $B_2(x_2)$ is contained in $B_1(x_1)$.
 We claim that the infimum $\eta(x_1)$ of all $x_2\in(a_2,b_2)$ such that $B_1(x_1)\subseteq B_2(x_2)$ lies in $(a_2,b_2)$.
 Indeed, otherwise we either had $B_2(x_2)\subseteq B_1(x_1)$ for all $x_2\in(a_2,b_2)$ or $B_1(x_1)\subseteq B_2(x_2)$ for all $x_2\in(a_2,b_2)$.
 In the first case this would mean that $B_1(x_1)$ is dense in $\LdBsm$, which is not possible in view of Corollary~\ref{cordBschrincl}.
  The second case would imply that for every function $F\in B_1(x_1)$ and $\zeta\in\C$ we have
 \begin{align*}
  |F(\zeta)|^2 & \leq \left|\dbspr{F}{K_2(\zeta,\cdot\,,x_2)}_{B_2(x_2)} \right|^2 
              \leq \|F\|_{B_2(x_2)}^2 \dbspr{K_2(\zeta,\cdot\,,x_2)}{K_2(\zeta,\cdot\,,x_2)}_{B_2(x_2)} \\
             & = \|F\|_{B_1(x_1)}^2 K_2(\zeta,\zeta,x_2) 
 \end{align*}
 for each $x_2\in(a_2,b_2)$.
 But since $K_2(\zeta,\zeta,x_2)\rightarrow0$ as $x_2\rightarrow a_2$ by~\eqref{eqndBschrRepKer}, we then had $B_1(x_1)=\lbrace 0\rbrace$, contradicting Theorem~\ref{thmdBschrBT}.
  Now from~\eqref{eqndBschrcontinuous} we infer that 
 \begin{align*}
   B_2(\eta(x_1)) = \overline{\bigcup_{x_2<\eta(x_1)} B_2(x_2)} \subseteq B_1(x_1) \subseteq \bigcap_{x_2>\eta(x_1)} B_2(x_2) = B_2(\eta(x_1))
 \end{align*}
 and hence $B_1(x_1)=B_2(\eta(x_1))$, including norms. 
 
 The function $\eta: (a_1,b_1)\rightarrow(a_2,b_2)$ defined this way is strictly increasing because of Corollary~\ref{cordBschrincl} and continuous by~\eqref{eqndBschrcontinuous}.
 Moreover, since for each  $\zeta\in\C$ we have 
  \begin{align*}
    K_2(\zeta,\zeta,\eta(x_1)) = K_1(\zeta,\zeta,x_1) \rightarrow 0,
  \end{align*}
  as $x_1\rightarrow a_1$, we infer that $\eta(x_1)\rightarrow a_2$ as $x_1\rightarrow a_1$.
  Finally,~\eqref{eqndBschrdense} shows that $\eta$ actually has to be a bijection.  
 Using the equation for the reproducing kernels~\eqref{eqndBschrRepKer} once more, we get for each $z\in\C$ 
 \begin{align*}
  \int_{a_1}^{x_1} |\phi_1(z,x)|^2 dx = \int_{a_2}^{\eta(x_1)} |\phi_2(z,x)|^2 dx, \quad x_1\in(a_1,b_1).
 \end{align*}
 Now by the implicit function theorem,  $\eta$ is differentiable (note that the integrand does not vanish if $z\in\C\backslash\R$) with
 \begin{align}\tag{$*$}\label{eqndBSMunique}
   |\phi_1(z,x_1)|^2 = \eta'(x_1) |\phi_2(z,\eta(x_1))|^2, \quad x_1\in(a_1,b_1),
 \end{align}
 which shows that $\eta$ is linear with $\eta'=1$ in view of the asymptotics in Lemma~\ref{lemSchrPhiAsym}. 
 %Taking quotients with different values of $x_1$ and using Lemma~\ref{lemSchrPhiAsym} twice shows %we infer from the equality 
% \begin{align*}
  %\left| \E^{2(x_1-\tilde{x}_1+\oo(1)) \sqrt{-z}}\right| & = 
%  \left|\frac{\phi_1(z,x_1)}{\phi_1(z,\tilde{x}_1)}\right|^2 
%           = \frac{\eta'(x_1)}{\eta'(\tilde{x}_1)} \left|\frac{\phi_2(z,\eta(x_1))}{\phi_2(z,\eta(\tilde{x}_1))}\right|^2, \quad x_1,\,\tilde{x}_1\in(a_1,b_1) 
  %        & = \left|\E^{2(\eta(x_1)-\eta(\tilde{x}_1)+\oo(1)) \sqrt{-z}}\right|,
% \end{align*}
 %that $\eta(x_1) - \eta(\tilde{x}_1) = x_1 - \tilde{x}_1$ for $x_1$, $\tilde{x}_1\in(a_1,b_1)$, i.e., $\eta$ is linear with gradient one. 

% Using~\eqref{eqndBSMunique} once more, we get for each $\lambda\in\R$
% \begin{align*}
%  \phi_1(\lambda,x_1)^2 = \phi_2(\lambda,\eta(x_1))^2, \quad x_1\in(a_1,b_1).
% \end{align*}
 Taking logarithmic derivatives in~\eqref{eqndBSMunique}, for each $\lambda\in\R$ we obtain 
 \begin{align}\tag{$**$}\label{eqndBSMuniquephiphi}
  \frac{\phi_1'(\lambda,x_1)}{\phi_1(\lambda,x_1)} = \frac{\phi_2'(\lambda,\eta(x_1))}{\phi_2(\lambda,\eta(x_1))}
 \end{align}
 for almost all $x_1\in(a_1,b_1)$.
 Differentiating this equation once more, we get
 \begin{align*}
  \frac{\phi_1''(\lambda,x_1)}{\phi_1(\lambda,x_1)} = \frac{\phi_2''(\lambda,\eta(x_1))}{\phi_2(\lambda,\eta(x_1))}
 \end{align*}
 for almost all $x_1\in(a_1,b_1)$ and thus also 
 \begin{align*}
 q_1(x_1)= \lambda + \frac{\phi_1''(\lambda,x_1)}{\phi_1(\lambda,x_1)} = \lambda + \frac{\phi_2''(\lambda,\eta(x_1))}{\phi_2(\lambda,\eta(x_1))} = q_2(\eta(x_1))
 \end{align*}
 for almost all $x_1\in(a_1,b_1)$.
  
 Finally note that~\eqref{eqndBSMuniquephiphi} implies that $\phi_1(\lambda,\cdot\,)$ and $\phi_2(\lambda,\eta(\cdot))$ are linearly dependent for each $\lambda\in\R$.
 In particular, if $\tau_1$ (and hence also $\tau_2$) is in the limit-circle case at the left endpoint, then this shows that the boundary conditions of $H_1$ and $H_2$ are the same there.
  Furthermore, if $\tau_1$ (and hence also $\tau_2$) is in the limit-circle case at the right endpoint, then $H_1$ and $H_2$ have some common eigenvalue $\lambda\in\R$. 
  Now the fact that $\phi_1(\lambda,\cdot\,)$ and $\phi_2(\lambda,\cdot\,)$ satisfy the respective boundary condition at the right endpoint shows that $H_1$ is equal to $H_2$ up to some shift.
\end{proof}

Note that even if one fixes the left endpoint, the operator is determined by the spectral measure in general only up to some shift.
This is due to the fact that we allowed the left endpoint to possibly be infinite. 
In fact, if one takes finite fixed left endpoints, the operators are uniquely determined by the spectral measure.

\begin{corollary}\label{cordBuniqnoshift}
 Suppose that $-\infty<a_1=a_2$ and that there is a real entire function $\E^g$ without zeros  such that
  \begin{align*}
   \E^{g(z)} \frac{E_1(z,x_1)}{E_2(z,x_2)}, \quad z\in\C^+
  \end{align*}
  is of bounded type for some $x_1\in (a_1,b_1)$ and $x_2\in(a_2,b_2)$.
  If $\rho_1=\E^{-2g} \rho_2$, then $b_1=b_2$, $q_1=q_2$ and $H_1=H_2$.
\end{corollary}

\begin{proof}
 This follows from Theorem~\ref{thmdBuniqS} and $\lim_{x_1\rightarrow a_1}\eta(x_1) = a_1$.
\end{proof}

Below we will see that some kind of growth restriction on the solutions $\phi_1$ and $\phi_2$ suffices to guarantee that~\eqref{eqnquotE1E2} is of bounded type.
However, note that this condition in Theorem~\ref{thmdBuniqS} can not be dropped and that some assumption has to be imposed on the solutions $\phi_1$ and $\phi_2$.
 As an example, consider the interval $(0,\pi)$, the potential $q_1=0$ and let $H_1$ be the associated Schr\"{o}dinger operator with Dirichlet boundary conditions.
 As our real entire solution $\phi_1$ we choose
 \begin{align*}
  \phi_1(z,x) = \frac{\sin\sqrt{z}x}{\sqrt{z}}, \quad x\in(0,\pi),~z\in\C.
 \end{align*}
 The associated spectral measure $\rho_1$ is given by
\begin{align*}
 \rho_1 = \frac{2}{\pi} \sum_{n\in\N} n^2 \delta_{n^2},
\end{align*}
 where for each $n\in\N$, $\delta_{n^2}$ is the Dirac measure in the point $n^2$.
 Now choose some sequence $\kappa_n$, $n\in\N$ of positive reals such that all but finitely many of these numbers are equal to one.
 From the solution of the inverse spectral problem in the regular case, it is known (see e.g.~\cite{levitan}, \cite{poestrub}) that there is some potential $q_2\in L^2(0,\pi)$ 
 and a corresponding operator $H_2$ with Dirichlet boundary conditions such that the spectral measure $\rho_2$ associated with the real entire solution $\phi_2$ of 
 \begin{align*}
   -\phi_2''(z,x) + q_2(x)\phi_2(z,x) = z\phi_2(z,x), \quad x\in(0,\pi),~z\in\C
 \end{align*}
 with the initial conditions $\phi_2(z,0) = 0$ and $\phi_2'(z,0) = 1$, $z\in\C$ is given by
 \begin{align*}
  \rho_2 = \frac{2}{\pi} \sum_{n\in\N} \kappa_n n^2 \delta_{n^2}.
 \end{align*}
 Now pick some real entire function $g$ such that
 \begin{align*}
  g(n^2) = \frac{\ln\kappa_n}{2}, \quad n\in\N
 \end{align*}
 and switch to the real entire solution
 \begin{align*}
  \tilde{\phi}_2(z,x) = \E^{g(z)} \phi_2(z,x), \quad x\in(0,\pi),~z\in\C.
 \end{align*}
 Then the spectral measure associated with this solution is equal to $\rho_1$, but the corresponding operators $H_1$ and $H_2$ are different (at least if not all $\kappa_n$ are equal to one).
 However, also note that in this case~\eqref{eqnquotE1E2} seems to fail to be of bounded type rather badly, since the function $\E^g$ is not even of finite exponential order.

We conclude this section by showing that condition~\eqref{eqnquotE1E2} in Theorem~\ref{thmdBuniqS} holds if the solutions $\phi_1$, $\phi_2$ satisfy some growth condition. 
Therefore, recall that an entire function $F$ belongs to the Cartwright class $\mathcal{C}$ if it is of finite exponential type and the logarithmic integral
\begin{align*}
 \int_\R \frac{\ln^+|F(\lambda)|}{1+\lambda^2}d\lambda < \infty
\end{align*}
exists, where $\ln^+$ is the positive part of the natural logarithm. In particular, note that the class $\mathcal{C}$ contains all entire functions of exponential order less than one.
Now a theorem of Krein~\cite[Theorem~6.17]{rosrov}, \cite[Section~16.1]{lev} states that the class $\mathcal{C}$ consists of all entire functions which are of bounded type in the upper and in the lower complex half-plane.
Since the quotient of two functions of bounded type is of bounded type itself, this immediately yields the following uniqueness result.

\begin{corollary}\label{cordBuniqScor}
 Suppose that $E_1(\,\cdot\,,x_1)$ and $E_2(\,\cdot\,,x_2)$ belong to the Cart\-wright class $\mathcal{C}$ for some $x_1\in(a_1,b_1)$ and $x_2\in(a_2,b_2)$. If $\rho_1=\rho_2$, then $H_1$ and $H_2$ are equal up to some shift.
\end{corollary}

 Again, as in Corollary~\ref{cordBuniqnoshift}, if one takes finite fixed left endpoints, then the operator is uniquely determined by the spectral measure.
 In particular, as a special case one recovers the classical result due to Borg and Marchenko that the spectral measure (corresponding to a particular choice of $\phi$) uniquely determines the operator, if the left endpoint is regular.
 Corollary~\ref{cordBuniqScor} shows that this continues to hold if the left endpoint is only assumed to be in the limit-circle case. 
 In fact, due to a result of Krein \cite{Kr52} it is still possible to choose $\phi$ such that the corresponding de Branges functions belong to the Cartwright class in this case. 

 However, our result is also applicable in situations where the left endpoint is also allowed to be in the limit-point case, as we will show in the next section. There we will apply our results in order to obtain an inverse uniqueness theorem for perturbed spherical Schr\"odinger operators.

\section{Perturbed spherical Schr\"odinger operators}

In this section we consider differential expressions of the form 
\begin{align*}
\tau = - \frac{d^2}{dx^2} + \frac{l(l+1)}{x^2} + q(x)
\end{align*}
on some interval $(0,b)$, where $l\in[-\frac{1}{2},\infty)$ and $q$ is some real-valued, locally integrable function on $(0,b)$ such that the function
\begin{align}\label{eqndBBesselqbar}
 \overline{q}(x) = \begin{cases}
                |q(x)| x, & l>-\frac{1}{2}, \\
                |q(x)| x(1-\ln x), & l=-\frac{1}{2},
              \end{cases}
\end{align}
 is integrable near zero. According to~\cite[Theorem~2.4]{kst}, $\tau$ is in the limit-circle case at zero if and only if $l\in[-\frac{1}{2},\frac{1}{2})$.
  With $H$ we denote some associated self-adjoint operator with the boundary condition
\begin{align}\label{eqndBBesselBC}
  \lim_{x\rightarrow 0} x^l ((l+1)f(x) - xf'(x)) = 0
\end{align}
 at zero, if necessary.
 In~\cite[Lemma~2.2]{kst} it has been shown that under these assumptions, there is a nontrivial real entire solution $\phi$ of exponential order one half which lies in $L^2(0,b)$ near zero and satisfies the boundary condition~\eqref{eqndBBesselBC} there, if any.
 Note that this solution is unique up to scalar multiples because of the growth restriction, as noted in~\cite[Lemma~6.4]{kst2}. 
 Consequently, the spectral measure $\rho$ associated with this solution $\phi$ is also unique up to a scalar multiple.

In order to state our inverse uniqueness theorem, let $l_1$, $l_2\in[-\frac{1}{2},\infty)$ and $q_1$, $q_2$ be two potentials on intervals $(0,b_1)$ respectively $(0,b_2)$, such that
 the functions $\overline{q}_1$, $\overline{q}_2$ defined as in~\eqref{eqndBBesselqbar} are integrable near zero.
 Furthermore, let $H_1$, $H_2$ be two corresponding self-adjoint operators with the boundary condition~\eqref{eqndBBesselBC} at zero, if necessary.
 With $\phi_1$, $\phi_2$ we denote some real entire solutions of exponential order one half which are square integrable near zero and satisfy the boundary condition there, if any.
 If $\rho_1$, $\rho_2$ are the corresponding spectral measures, then the uniqueness results from the preceding section yield the following theorem.

\begin{theorem}
If $\rho_1=\rho_2$, then $l_1=l_2$, $b_1=b_2$, $q_1=q_2$ and $H_1=H_2$.
\end{theorem}

\begin{proof}
 Since the solutions are of exponential order one half, we may immediately apply Corollary~\ref{cordBuniqScor} and obtain $b_1=b_2$, 
 \begin{align*}
  \frac{l_1(l_1+1)}{x^2} + q_1(x) = \frac{l_2(l_2+1)}{x^2} + q_2(x)
 \end{align*}
  for almost all $x\in(0,b_1)$ and $H_1=H_2$.
  Now since the function
  \begin{align*}
   xq_1(x) - xq_2(x) = \frac{l_2(l_2+1) - l_1(l_1+1)}{x}, \quad x\in(0,b_1)
  \end{align*}
  is integrable near zero, we infer $l_1(l_1+1)=l_2(l_2+1)$ and hence $l_1=l_2$.
\end{proof}

\bigskip
\noindent
{\bf Acknowledgments.}
I thank Fritz Gesztesy, Gerald Teschl and Harald Woracek for helpful discussions and hints with respect to the literature.

\end{document}